\documentclass[11pt,twoside]{amsart}
\usepackage[english]{babel}
\usepackage{amsmath, amsthm, amssymb, amsfonts}
\usepackage{color}
\usepackage{latexsym}
\usepackage{graphicx}
\usepackage{pdfpages}
\usepackage {bm}
\usepackage {indentfirst} %indent the first par after section

\usepackage{hyperref,cite}

\advance\hoffset by -1.5cm

\newcommand{\N}{\mathbb{N}}

\newcommand{\Z}{\mathbb{Z}}
\newcommand{\er}{\mathbb{R}}
\newcommand{\B}{\mathcal{B(H)}}
\newcommand{\R}{\mathcal{R}}
\newcommand{\n}{\mathcal{N}}
\newcommand{\h}{\mathcal{H}}

\newcommand{\F}{F{\o}lner }

  \newcommand{\norm}[1]{\left\| #1 \right\|}
  \newcommand{\abs}[1]{\left\vert #1 \right\vert}
  \newcommand{\ps}[2]{\left\langle #1,#2 \right\rangle}
 \newcommand{\inserteq}[2]{\begin{equation}{#1}\label{#2}\end{equation}}
 
 \newcommand{\sumi}{\int}
 \newcommand{\dd}[1]{\, \mathrm{d}m(#1)}

\newtheorem{theorem}{Theorem}[section]

\newtheorem{corollary}[theorem]{Corollary}
\theoremstyle{definition}

\theoremstyle{remark}
\newtheorem{remark}[theorem]{Remark}
\numberwithin{equation}{section}

\numberwithin{equation}{section}

\begin{document}

\title[F{\o}lner sets and isometric representations of semigroups]{Similarity problems, F{\o}lner sets and isometric representations of amenable semigroups}

 \author[C. Badea]{Catalin Badea}
  % \address{D\'{e}partement de Math\'{e}matiques, Laboratoire Paul Painlev\'{e}, UMR CNRS 8524, Universit\'{e} de Lille, 59655 Villeneuve d'Ascq, France}
      \address{Univ. Lille, CNRS, UMR 8524 - Laboratoire Paul Painlev\'e, Lille, France} 
 \email{catalin.badea@univ-lille.fr}
\author[L. Suciu]{Laurian Suciu}
 \address{Department of Mathematics and Informatics, ``Lucian Blaga'' University of Sibiu, Dr. Ion Ra\c{t}iu 5-7, Sibiu, 550012, Romania}
 \email{laurians2002@yahoo.com}
 \thanks{This work was supported in part by
 the project FRONT of the French
National Research Agency (grant ANR-17-CE40-0021) and by the Labex CEMPI (ANR-11-LABX-0007-01).}

\begin{abstract}
We revisit Sz.-Nagy's criteria for similarity of Hilbert space bounded linear operators to isometries or unitaries and present new ones. We also discuss counterparts of the Dixmier-Day theorem concerning bounded representations of amenable groups and semigroups. We highlight the role of \F sets in similarity problems in both settings of unimodular, $\sigma$-compact, amenable groups and in discrete semigroups possessing the Strong \F condition (SFC).
\end{abstract}
\maketitle
\section{Introduction}
\medskip
A bounded linear operator $T\in\B$ acting on a Hilbert space $\h$ is said to be similar to a unitary (or to an isometry) if $L^{-1}TL$ is a unitary operator (or an isometry) for some invertible operator $L\in \B$. The first criterion of similarity to a unitary operator was obtained by B\'ela Sz.-Nagy \cite{Nagy} in $1947$: a Hilbert space operator $T$ is similar to a unitary operator if and only if $T$ is invertible and both $T$ and its inverse $T^{-1}$ are power-bounded, that is $\sup_{n\in\Z}\norm{T^n} < \infty$. The simple, yet ingenious proof given by Sz.-Nagy uses a Banach limit $\mathcal{L}$, i.e. a positive linear functional of norm one on $\ell^{\infty}(\Z)$, invariant by translations, which extends to bounded sequences the classical limit of convergent sequences. Starting from an operator $T$ with both $T$ and $T^{-1}$ power-bounded, one can consider a new norm
$$ \abs{x}^2 = \mathcal{L} \left(\left(\norm{T^nx}^2\right)_{n\in\Z}\right) .$$
Then $\abs{\cdot}$ is an equivalent norm, coming from an inner product, and the invariance by translations property of $\mathcal{L}$ implies that $\abs{Tx} = \abs{x}$. Thus $T$ is similar to an invertible isometry, i.e. a unitary operator. Basically the same proof shows that $T\in\B$ is similar to an isometry if (and only if) there are two positive constants $m$ and $M$ such that
\begin{equation}\label{eq:0}
m\norm{x} \le \norm{T^nx} \le M\norm{x} \quad \textrm{ for every } n\ge 1  \textrm{ and every } x\in \h .
\end{equation}

The existence of Banach limits, first proved by Mazur and Banach, is a demonstration of the amenability of the group $\Z$. What are now called amenable groups were introduced in 1929 by von Neumann \cite{vN}. The term ``amenable'' was coined by Day \cite{DayIllinois}, who also broadened consideration to encompass semigroups. We refer for instance to the book \cite{Paterson} for more information about amenability of (semi)groups. It was independently remarked in 1950 by Dixmier \cite{Dixmier} and Day \cite{Day}, and in 1951 by Nakamura and Takeda \cite{NaTa}, that a variant of Sz.-Nagy's proof shows that uniformly bounded representations of amenable groups into $\B$ are \emph{unitarizable}: they are unitary representations in an equivalent Hilbertian norm.  One can also  consult \cite{PisierSurvey} for a survey about the converse question (Is a group amenable if all its bounded representations are unitarizable~?), which is presently still open in full generality.

The aim of this note is to revisit Sz.-Nagy's similarity results, the Dixmier-Day theorem and its analogue for ``isometrizable'' representations of amenable semigroups. Instead of using a Banach limit we will use a limit along a non-principal ultrafilter. A natural question to address in the case of the semigroup $\N$ is if it is true that we can replace the condition \eqref{eq:0} in Sz.-Nagy's criterion by a similar boundedness condition in terms of Ces\`{a}ro means, namely
$$m^2\norm{x}^2 \le \frac{1}{N}\sum_{n=0}^{N-1}\norm{T^nx}^2 \le M^2\norm{x}^2  \quad \textrm{ for every } N\ge 1  \textrm{ and every } x\in \h .$$
It is one of the aims of this note to give a positive answer to this question (see Corollary \ref{cor:ces}). This is obtained as a consequence of the following general result.
\begin{theorem}\label{thm:K}
Let $T\in \B$. Suppose there exist positive constants $m$ and $M$ and an operator $K\in \B$ with closed range and finite-dimensional kernel such that:
\begin{equation}\label{eq:1}
m\norm{Kx}^2 \le \frac{1}{N}\sum_{n=0}^{N-1}\norm{T^nx}^2 \le M\norm{x}^2
\end{equation}
for every $N\ge 1$ and every $x\in \h$
and
\begin{equation}\label{eq:3}
\text{every eigenvalue of } T\text{, if any, has modulus one.}
\end{equation}
Then $T$ is similar to an isometry.
\end{theorem}
Note that, conversely, if $T$ is similar to an isometry, then \eqref{eq:1} is satisfied with $K=I$ (the identity operator)
and \eqref{eq:3} is readily verified. Theorem \ref{thm:K} is a generalization of a result from \cite{Pruv} (see also \cite{Pruv2}).
In the above theorem, \eqref{eq:3} follows for instance from the condition that for any non zero $x\in\h$, the closure
of the orbit $ \{T^nx : n\ge 0\}$ of $x$ does not contain the zero vector. When $K=I$, the condition \eqref{eq:3} follows from \eqref{eq:1} (see the proof of Corollary~\ref{cor:ces}).

Considering Ces\`{a}ro means as in \eqref{eq:1} leads to considering \F sets in amenable groups and semigroups. Note that, contrary to the case of groups, amenability of semigroups is not characterized by \F type conditions, or indeed by any known elementary combinatorial condition. We discuss here (Theorem \ref{thm:SFC}) an analogue of the Dixmier-Day theorem
for discrete semigroups satisfying the strong \F condition (SFC) of Argabright and Wilde \cite{Arga}. We also consider formally weaker conditions implying unitarizability in the case of unimodular, $\sigma$-compact amenable groups. Some of these results generalize results from \cite{BerksonG,vC1,vC2} (for $G=\Z$) and from \cite{GuoZwart} (for $G=\er$); see Theorem \ref{thm:cs} and Theorem \ref{thm:extra}.

\section{Unitarizable representations of unimodular amenable groups}
Let $G$ be a $\sigma$-compact unimodular amenable group. We denote by $m$ a right invariant Haar measure on $G$ and write $\abs{E}$ instead of $m(E)$ for a measurable $E\subset G$. In this note, by a \emph{symmetric \F sequence} we understand a sequence of \emph{symmetric} compact sets $(F_N)_{N\ge 1}$ with $0 < \abs{F_N} < \infty$ and $F_N^{-1} = F_N$ and which is a (right) \F sequence for $G$, that is
\begin{equation}\label{eq:f}
\lim_{N\to\infty} \frac{\abs{(F_{N}s\, \triangle F_N)}}{\abs{F_N}} = 0
\end{equation}
for every $s\in G$. Here $A^{-1} = \{a^{-1} : a\in A\}$, $As = \{as : a\in A\}$ and $A\, \triangle B = (A\setminus B)\cup (B\setminus A)$ is the symmetric difference of $A$ and $B$. The existence of symmetric \F sets has been proved by Namioka \cite{Namioka} for countable amenable groups and by Emerson \cite{Emerson} in the general case. We may note in passing that Tessera proved in \cite{Tessera} that the property of being amenable and unimodular is invariant under large scale equivalence between $\sigma$-compact locally compact groups. We also note that a good part of the (large) literature concerning amenable groups deals with left invariant Haar measure, left \F sequences, etc. When dealing with variants of the Dixmier-Day theorem for groups and semigroups we found more convenient to work with right \F sequences. The results obtained in this section have counterparts for $\sigma$-compact amenable groups, \emph{i.e.}, without assuming unimodularity. We leave the task of writing these counterparts to the interested reader.

A \emph{continuous representation} $\pi$ of $G$ will be a map $\pi : G \mapsto \B$ from $G$ to the $C^*$-algebra of all bounded operators on $\h$ such that $\pi(1) = I$ and $\pi(g_1g_2) = \pi(g_1)\pi(g_2)$ of every $g_1,g_2\in G$ and that is continuous for the strong operator topology on $\B$. We say that $\pi$ is an \emph{unitary representation} if $\pi(g)^{\ast}\pi(g) = I = \pi(g)\pi(g)^{\ast}$ for any $g\in G$.

\begin{theorem}\label{thm:amenthm}
Let $G$ be a $\sigma$-compact unimodular amenable group generated by $Q\subset G$ and suppose that $(F_N)_{N\ge 1}$ is a symmetric \F sequence. Let $\pi : G \mapsto \B$ be a (strongly) continuous representation of $G$ by invertible operators on a complex Hilbert space $\h$. Suppose that there is a constant $C\ge 1$ such that
\begin{equation}\label{eq:2a}
\frac{1}{\abs{F_N}} \sumi_{F_N}\norm{\pi(g)(x)}^2\dd{g} \le C^2\norm{x}^2
\end{equation}
for every $x\in\h$,
\begin{equation}\label{eq:2b}
\frac{1}{\abs{F_N}} \sumi_{F_N}\norm{\pi(g)^{\ast}(x)}^2\dd{g} \le C^2\norm{x}^2
\end{equation}
for every $x\in\h$, and that
\begin{equation}\label{eq:22}
\lim_{N\to\infty}\frac{1}{\abs{F_N}}  \sumi_{(F_{N}s\, \triangle F_N)}\norm{\pi(g)}^2\dd{g} = 0
\end{equation}
for every $s\in Q$.
Then $\pi$ is unitarizable and there exists an invertible $L\in\B$ with $\|L^{-1}\| \|L\| \le C^2$ such that $L^{-1}\pi L$ is a unitary representation of $G$.
\end{theorem}
In the case when $\pi$ is a bounded representation with
$$\sup_{g\in G} \norm{\pi(g)} \le M ,$$
the inequalities \eqref{eq:2a} and \eqref{eq:2b} are verified with $C=M$, while the condition \eqref{eq:22} holds true since $(F_N)_{N\ge 1}$ is a right \F sequence. Therefore Theorem \ref{thm:amen} implies, when $G$ is unimodular, the Dixmier-Day theorem.

\begin{proof}[Proof of Theorem \ref{thm:amenthm}]
Let $x,y\in \h$. We have
\begin{align*}
\abs{\ps{x}{y}} &= \frac{1}{\abs{F_N}}\abs{\ps{x}{\sumi_{F_N}\pi(g)\pi(g^{-1})y\dd{g}}} \\
& = \frac{1}{\abs{F_N}} \abs{\sumi_{F_N}\ps{\pi(g)^{\ast}x}{\pi(g^{-1})y}\dd{g}}\\
& \le \frac{1}{\abs{F_N}} \sumi_{F_N} \norm{\pi(g)^{\ast}x}\norm{\pi(g^{-1})y} \dd{g} &\quad\text{(Cauchy-Schwarz)}\\
& \le \left(\frac{1}{\abs{F_N}}\sumi_{F_N}\norm{\pi(g)^{\ast}x}^2 \dd{g}\right)^{1/2} \left(\frac{1}{\abs{F_N}}\sumi_{F_N}\norm{\pi(g^{-1})y}^2 \dd{g}\right)^{1/2} &\quad\text{(Cauchy-Schwarz)}\\
& \le C\norm{x} \left(\frac{1}{\abs{F_N}}\sumi_{F_N}\norm{\pi(g^{-1})y}^2 \dd{g}\right)^{1/2} &\quad \text{(using }\eqref{eq:2b}).
\end{align*}
Taking the supremum after all $x$ in the unit ball we obtain
\begin{equation*}
\frac{1}{\abs{F_N}}\sumi_{F_N}\norm{\pi(g^{-1})y}^2 \dd{g} \ge \frac{1}{C^2} \norm{y}^2
\end{equation*}
for every $y\in \h$.
As $G$ is unimodular, the Haar measure is inverse invariant. Using also the symmetry condition $F_N^{-1} = F_N$, we obtain
\begin{equation}\label{eq:lower}
\frac{1}{\abs{F_N}}\sumi_{F_N}\norm{\pi(g)y}^2 \dd{g} \ge \frac{1}{C^2} \norm{y}^2
\end{equation}
for every $y\in \h$.

Let $\mathcal{U}$ be a non-principal ultrafilter on $\N$.  Define a new norm
\begin{equation}\label{eq:defnorm}
\abs{x}^2 = \lim_{\mathcal{U}}  \left(\frac{1}{\abs{F_N}}\sumi_{F_N} \norm{\pi(g)x}^2 \dd{g}\right) ,
\end{equation}
by taking the limit along $\mathcal{U}$ of the bounded sequence in \eqref{eq:defnorm} indexed by $N$.
Using \eqref{eq:2a} and \eqref{eq:lower} we have
\begin{equation}\label{eq:new2}
\frac{1}{C} \norm{x} \le \abs{x} \le C\norm{x}.
\end{equation}
Thus $\abs{\cdot}$ is an equivalent norm. It is also a Hilbertian norm for the inner product
$$ \lim_{\mathcal{U}}  \left( \frac{1}{\abs{F_N}}\sumi_{F_N} \ps{\pi(g)x}{\pi(g)y} \dd{g}\right) .
$$
Let $s\in Q$.
We can write
\begin{equation}\label{eq:nisom}
\abs{\pi(s)x}^2 -  \abs{x}^2 = \lim_{\mathcal{U}}  \left( \frac{1}{\abs{F_N}}\sumi_{F_N} \norm{\pi(gs)x}^2\dd{g} -  \frac{1}{\abs{F_N}}\sumi_{F_N} \norm{\pi(g)x}^2\dd{g}\right) .
\end{equation}
We have
\begin{align*}
& \abs{\frac{1}{\abs{F_N}}\sumi_{F_N} \norm{\pi(gs)x}^2\dd{g} -  \frac{1}{\abs{F_N}}\sumi_{F_N} \norm{\pi(g)x}^2\dd{g}} \\
&= \abs{\frac{1}{\abs{F_N}}\sumi_{F_Ns} \norm{\pi(h)x}^2\dd{h} -  \frac{1}{\abs{F_N}}\sumi_{F_N} \norm{\pi(g)x}^2\dd{g}}\\
&= \abs{\frac{1}{\abs{F_N}}\sumi_{F_{N}s\setminus F_N} \norm{\pi(h)x}^2\dd{h} -  \frac{1}{\abs{F_N}}\sumi_{F_{N}\setminus F_{N}s} \norm{\pi(g)x}^2\dd{g}}\\
&\le \frac{1}{\abs{F_N}} \sumi_{F_{N}s\triangle F_N} \norm{\pi(g)x}^2 \dd{g} .
\end{align*}
Using \eqref{eq:22} we obtain that the limit in \eqref{eq:nisom} is zero. Thus $\pi(s)$ is unitary (an invertible isometry) in the new norm for each $s\in Q$. As $Q$ generates $G$, $\pi(g)$ is a unitary in the new norm for every $g\in G$. Therefore $\pi$ is unitarizable. The statement about the similarity constant follows from \eqref{eq:new2} in the same way as in the proof of the Dixmier-Day theorem.
\end{proof}

In some cases, we do not need condition \eqref{eq:22} in Theorem \ref{thm:amenthm}. Under some additional conditions on \F sets, all representations satisfying \eqref{eq:2a} and \eqref{eq:2b} are already uniformly bounded and thus \eqref{eq:22} follows from the \F condition. 
%It is an open problem to characterize  \F sets with this property. 
Here is an example of such additional conditions, generalizing results from \cite{BerksonG} for $G=\Z$ and $F_N = \{-N, -N+1, \cdots , N-1, N\}$ (see also \cite{vC1,vC2}) and from \cite{GuoZwart} for $G=\er$. 
%Other variations are possible.
\begin{theorem}\label{thm:cs}
Let $G$ be a $\sigma$-compact unimodular amenable group and suppose that $(F_N)_{N\ge 1}$ is a symmetric \F sequence such $\cup_{N\ge 1} F_N = G$ and there exist $p\in \N\setminus \{0\}$ and $K> 0$ such that
\begin{equation}\label{eq:incl}
F_N\cdot F_N^{-1} \subset F_{pN}
\end{equation}
and
\begin{equation}\label{eq:meas}
\frac{\abs{F_{pN}}}{\abs{F_N}} \le K
\end{equation}
for every $N$.
Then every strongly continuous representation of $G$ by invertible operators on a complex 
Hilbert space $\h$ satisfying \eqref{eq:2a} and \eqref{eq:2b} is unitarizable.
%a uniformly bounded representation.
\end{theorem}
\begin{proof}
We apply \eqref{eq:2a} to obtain
$$ \frac{1}{\abs{F_{pN}}} \sumi_{g\in F_{pN}}\norm{\pi(g)(x)}^2\dd{g} \le C^2\norm{x}^2$$
for every $x\in\h$. Replace $x$ in this inequality by $\pi(k)x$, where $k\in F_N$. Then 
$$\frac{1}{\abs{F_{pN}}} \sumi_{g\in F_{pN}}\norm{\pi(g)\pi(k)x}^2\dd{g} \le C^2\norm{\pi(k)x}^2$$
and so 
$$ \frac{1}{\abs{F_{pN}}} \sumi_{g\in F_{pN}k}\norm{\pi(g)(x)}^2\dd{g} \le C^2\norm{\pi(k)x}^2 .$$
As $F_N\cdot F_N^{-1} \subset F_{pN}$ we have $F_N \subset F_{pN}k$. Therefore 
$$ \sumi_{g\in F_{N}}\norm{\pi(g)(x)}^2\dd{g} \le \abs{F_{pN}}C^2\norm{\pi(k)x}^2 .$$
Thus, using \eqref{eq:meas}, we have 
$$ \frac{1}{\abs{F_{N}}} \sumi_{g\in F_{N}}\norm{\pi(g)(x)}^2\dd{g} \le KC^2\norm{\pi(k)x}^2$$
for every $x\in\h$. Using now \eqref{eq:lower}, we get $\norm{\pi(k)x} \ge \frac{1}{C^2\sqrt{K}}\norm{x}$ for every $x\in\h$ and every $k\in F_N$. Thus $\pi$ is a uniformly bounded representation, and so unitarizable.
\end{proof}
When the \F sets $F_N$ are balls with respect to an invariant metric, \eqref{eq:incl} is satisfied with $p=2$ and \eqref{eq:meas} is a doubling condition. 
See \cite[\S 6.2]{nevo} for a discussion of such \F sets. \F sets $(F_N)$ satisfying \eqref{eq:incl} and \eqref{eq:meas} satisfy the Tempelman condition 
\begin{equation}\label{eq:temp}
\abs{F_{N}^{-1}\cdot F_{N}} \le \textrm{const} \abs{F_{N}}
\end{equation}
for every $N$. 

Another instance where we do not need condition \eqref{eq:22} in Theorem \ref{thm:amenthm} is given in the following result. This time we assume a Ces\`aro-type boundedness condition on translations of \F sets.

\begin{theorem}\label{thm:extra}
Let $G$ be a $\sigma$-compact unimodular amenable group and suppose that $(F_N)_{N\ge 1}$ is a symmetric \F sequence. Let $\pi : G \mapsto \B$ be a (strongly) continuous representation of $G$ by invertible operators on a complex Hilbert space $\h$. Suppose that there is a constant $C\ge 1$ such that
\begin{equation}\label{eq:2aa}
\frac{1}{\abs{F_Ng}} \sumi_{F_Ng}\norm{\pi(h)(x)}^2\dd{h} \le C^2\norm{x}^2
\end{equation}
for every $x\in\h$ and every $g\in G$ and
\begin{equation}\label{eq:2bb}
\frac{1}{\abs{F_N}} \sumi_{F_N}\norm{\pi(h)^{\ast}(x)}^2\dd{h} \le C^2\norm{x}^2
\end{equation}
for every $x\in\h$. 
Then $\pi$ is unitarizable and there exists an invertible $L\in\B$ with $\|L^{-1}\| \|L\| \le C^2$ such that $L^{-1}\pi L$ is a unitary representation of $G$.
\end{theorem}
\begin{proof}
Let $g\in G$ and $x,y\in \h$. We have 
\begin{align*}
& \abs{\ps{\pi(g)x}{y}} = \frac{1}{\abs{F_N}} \sumi_{F_N}\abs{\ps{\pi(g)x}{y}}\dd{h}\\
 & =  \frac{1}{\abs{F_N}} \sumi_{F_N} \abs{\ps{\pi(h^{-1}g)x}{\pi(h)^{\ast}y}}\dd{h}\\
& \le  \frac{1}{\abs{F_N}} \sumi_{F_N} \norm{\pi(h^{-1}g)x} \, \norm{\pi(h)^{\ast}y}\dd{h}\\
& \le \left( \frac{1}{\abs{F_N}} \sumi_{F_N} \norm{\pi(h^{-1}g)x}^2\dd{h}\right)^{1/2} \left( \frac{1}{\abs{F_N}} \sumi_{F_N} \norm{\pi(h)^{\ast}y}^2\dd{h}\right)^{1/2}.
\end{align*}
Using the symmetry of $F_N$, the unimodularity of $G$ and the right invariance of the Haar measure we obtain
$$
\abs{\ps{\pi(g)x}{y}} \le \left( \frac{1}{\abs{F_Ng}} \sumi_{F_Ng} \norm{\pi(u)x}^2\dd{u}\right)^{1/2} \left( \frac{1}{\abs{F_N}} \sumi_{F_N} \norm{\pi(h)^{\ast}y}^2\dd{h}\right)^{1/2}.
$$
Now the boundedness conditions \eqref{eq:2aa} and \eqref{eq:2bb} give 
$$ \abs{\ps{\pi(g)x}{y}} \le C^2\norm{x}\norm{y}.$$
Thus $\pi$ is a uniformly bounded representation and so it is unitarizable.
\end{proof}
We obtain the following consequence. 
\begin{corollary}
\label{thm:amen}
Let $G$ be a $\sigma$-compact unimodular amenable group and suppose that $(F_N)_{N\ge 1}$ is a symmetric \F sequence. Let $\pi : G \mapsto \B$ be a (strongly) continuous representation of $G$ by invertible operators on a complex Hilbert space $\h$. Suppose that there is a constant $C\ge 1$ such that for every $g\in G$ we have
\begin{equation}\label{eq:21}
\frac{1}{\abs{F_Ng}} \sumi_{h\in F_Ng}\norm{\pi(h)}^2\dd{h} \le C^2.
\end{equation}
Then $\pi$ is unitarizable and there exists an invertible $L\in\B$ with $\|L^{-1}\| \|L\| \le C^2$ such that $L^{-1}\pi L$ is a unitary representation of $G$.
\end{corollary}
\begin{proof}
As an operator and its adjoint have the same norm, \eqref{eq:21} implies \eqref{eq:2aa} and \eqref{eq:2bb} (for $g$ being the identity element). We apply Theorem \ref{thm:amenthm}.
\end{proof}
The following corollary concerns inner derivations. Recall that, for a group $G$ and a unitary representation $\pi : G \to \B$, the map $D : G \mapsto \B$ is called a \emph{derivation}, or a $\pi$-derivation, if it satisfies the Leibniz rule $D(gh) = D(g)\pi(h) + \pi(g)D(h)$. The derivation $D$ is said to be \emph{inner} if there is $T\in\B$ such that $D(g) = \pi(g)T - T\pi(g)$ for every $g$. 
\begin{corollary}\label{cor:deriv}
Let $G$ be a $\sigma$-compact unimodular amenable group and suppose that $(F_N)_{N\ge 1}$ is a symmetric \F sequence. Let $\pi : G \mapsto \B$ be a continuous unitary representation of $G$ and let $D : G \mapsto \B$ be a $\pi$-derivation. Suppose that there is a constant $C\ge 1$ such that for every $g\in G$ we have
\begin{equation}\label{eq:2aderiv}
\frac{1}{\abs{F_Ng}} \sumi_{F_Ng}\norm{D(h)}^2\dd{h} \le C^2.
\end{equation}
%and that
%\begin{equation}\label{eq:22deriv}
%\lim_{N\to\infty}\frac{1}{\abs{F_N}}  \sumi_{(F_{N}s\, \triangle F_N)}\norm{D(g)}^2\dd{g} = 0
%\end{equation}
%for every $s\in Q$. 
Then $D$ is inner.
\end{corollary}
\begin{proof}
We apply Corollary \ref{thm:amen} to the representation
$$ \pi_D(g) = \begin{bmatrix}
    \pi(g)       & D(g) \\
    0       &  \pi(g)
\end{bmatrix} \in \mathcal{B}(\h\oplus\h) .$$
Then $\pi_D$ is unitarizable and therefore $D$ is inner (see \cite[Lemma 4.5]{PisierBook}).
\end{proof}

\section{Dixmier-Day type theorems for semigroups}
Let us now turn to versions of the Dixmier-Day theorem for semigroups. We first record a version of the Dixmier-Day theorem for right-amenable locally compact semigroups. The proof of this result is similar to the case of groups and it is surely known to specialists. The proof can be also extracted from \cite[Prop. 2]{Fack}. However, as several proofs in the case of amenable groups use left invariant means and the setting in \cite{Fack} is that of left and right amenable semigroups (and of finite type representations), we decided to present a complete (short) proof. In the second part of this section we obtain a counterpart of the Dixmier-Day theorem for discrete semigroups possessing the strong \F condition (SFC) from \cite{Arga}. This is similar in spirit to Theorem \ref{thm:amenthm} for unimodular amenable groups.

\subsection{Right-amenable semigroups} Let $S$ be a \emph{locally compact semigroup}, that is a locally compact Hausdorff space $S$ endowed with an associative and separately continuous multiplication law. For $s\in S$ and $f$ in the algebra $CB(S)$ of all bounded continuous complex functions on $S$ we define $r_s(f) \in CB(S)$ by $r_sf(t) = f(ts)$. By definition, a \emph{right invariant mean} on $S$ is a state $m_r$ on $CB(S)$ such that $m_r(r_sf) = m_r(f)$ for any $f\in CB(S)$ and any $s\in S$. A locally compact semigroup $S$ is said to be \emph{right amenable} if there exists a right invariant mean on $S$. A \emph{continuous representation} $\pi$ of $S$ will be a map $\pi : S \mapsto \B$ from $S$ to the C$^*$-algebra of all bounded operators on $\h$ such that $\pi(st) = \pi(s)\pi(t)$ of every $s,t\in S$ and that is continuous for the strong operator topology on $\B$. We say that $\pi$ is an \emph{isometric representation} if $\pi(s)^{\ast}\pi(s) = I$ for any $s\in S$.

We record here the known analogue for right-amenable semigroups of the Dixmier-Day theorem.

\begin{theorem}\label{thm:right}
Let $S$ be a right amenable locally compact semigroup and let $\pi : S \mapsto \B$ be a continuous representation such that 
\begin{equation}\label{eq:bb}
m^2I \le \pi(s)^{\ast}\pi(s) \le M^2I \quad (s\in S)
\end{equation}
for some positive constants $m$ and $M$.
Then there exists an invertible operator $L\in \B$ with $\|L^{-1}\|\|L\| \le Mm^{-1}$ such that $L^{-1}\pi(\cdot)L$ is an isometric representation.
\end{theorem}

\begin{proof}
For every $x$ and $y$ in $\h$, let
$ f_{x,y}(t) = \ps{\pi(t)x}{\pi(t)y} .$
Using the Cauchy-Schwarz inequality we obtain $f_{x,y}\in CB(S)$. Denoting by $m_r$ the right invariant mean on $S$, we consider the inner product given by $m_r(f_{x,y})$ and the induced norm
$\abs{x}^2 = m_r(f_{x,x}).$
Then
$ m^2\norm{x}^2 \le \abs{x}^2 \le M^2\norm{x}^2 $
and $\abs{\cdot}$ is an equivalent norm.
As $\pi$ is a representation of $S$ we have
$ r_sf_{x,y} = f_{\pi(s)x,\pi(s)y}$
for any $s\in S$ and any $x$ and $y$. In the new norm $\abs{\cdot}$, the homomorphism $\pi$ is an isometric representation. Indeed, using the right invariance of $m_r$, we have
$ \abs{\pi(s)x}^2 = m_r(f_{\pi(s)x,\pi(s)x})  = m_r(r_sf_{x,x}) = m_r(f_{x,x}) = \abs{x}^2.$
The statement about the similarity constant follows from \eqref{eq:bb} in the same way as in the proof of the Dixmier-Day theorem.
\end{proof}

\subsection{Strong \F condition and ``isometrizable'' representations}
Suppose now that $S$ is a discrete semigroup. In this subsection $\abs{E}$ denotes the cardinality of the finite set $E\subset S$. The semigroup $S$ is said to satisfy \emph{the \F condition} (\emph{FC})
if for every finite subset $H$ of $S$ and every $\epsilon > 0$, there
is a finite subset $F$ of $S$ with $|Fs \setminus F| \leq \epsilon |F|$ for
all $s \in H$.
A semigroup $S$ satisfies \emph{the strong \F condition} (\emph{SFC})
if for every finite subset $H$ of $S$ and every $\epsilon > 0$, there
is a finite subset $F$ of $S$ with $|F \setminus Fs| \leq \epsilon |F|$
for all $s \in H$. It is known that
SFC implies right amenability (\cite{Arga}), which in turn
implies FC (\cite{DayIllinois}), but neither of these implications is reversible in general (\cite{DayIllinois,Klawe77}). However, for many semigroups (groups, right
cancellative semigroups, finite semigroups, compact topological semigroups, inverse semigroups, regular semigroups, commutative semigroups and semigroups with a left, right or two-sided zero element), right amenability coincides with the strong \F condition (SFC); see \cite{GrayKambites} and the references therein.

The direct proof that (SFC) implies (FC) is simple. Indeed, for any finite set $F\subset S$ and any $s\in S$ we have $\abs{F\setminus Fs} = \abs{F} - \abs{F\cap Fs}$ and $\abs{Fs\setminus F} = \abs{Fs} - \abs{F\cap Fs}$. Therefore
$$ \abs{F\setminus Fs} - \abs{Fs\setminus F} = \abs{F} - \abs{Fs} \ge 0$$
and so (SFC) implies (FC). This also implies the following reformulation: the (SFC) condition holds for the semigroup $S$ if and only if for every finite subset $H$ of $S$ and every $\epsilon > 0$, there
is a finite subset $F$ of $S$ with $\abs{Fs \,\triangle F} \leq \epsilon \abs{F}$
for all $s \in H$.
Moreover, if $S$ is countable, it follows that (SFC) is equivalent to the existence of a right \F sequence, that is a sequence of finite sets $(F_N)_{N\ge 1}$ such that
$$ \lim_{N\to\infty}\frac{\abs{F_{N}s\, \triangle \, F_N}}{\abs{F_N}} = 0 $$
for each $s\in S$.

\begin{theorem}\label{thm:SFC}
Let $S$ be a countable, discrete semigroup satisfying (SFC) and let $(F_N)_{N\ge 1}$ be a right \F sequence. Let $\pi : S \mapsto \B$ be a representation of $S$ such that
\begin{equation}\label{eq:bc}
m^2I \le \frac{1}{\abs{F_N}}\sum_{g\in F_N}\pi(g)^{\ast}\pi(g) \le M^2I 
\end{equation}
for some positive constants $m$ and $M$ and
\begin{equation}\label{eq:bd}
\lim_{N\to\infty}\frac{1}{\abs{F_N}}  \sum_{g\in (F_{N}s \, \triangle F_N)}\norm{\pi(g)}^2 = 0
\end{equation}
for any $s\in S$.
Then there exists an invertible operator $L\in \B$ with $\|L^{-1}\|\|L\| \le Mm^{-1}$ such that $L^{-1}\pi(\cdot)L$ is an isometric representation.
\end{theorem}

In the case when $\pi$ is a representation satisfying \eqref{eq:bb}, the inequality \eqref{eq:bc} is also readily verified, while the condition \eqref{eq:bd} holds true since $(F_N)_{N\ge 1}$ is a right \F sequence. Therefore Theorem \ref{thm:SFC} implies Theorem \ref{thm:right} for countable, discrete semigroups satisfying (SFC).

\begin{proof}
We use the same idea as in Theorem \ref{thm:amen}.
Let $\mathcal{U}$ be a non-principal ultrafilter on $\N$.  Define a new norm
\begin{equation}\label{eq:SFCnorm}
\abs{x}^2 = \lim_{\mathcal{U}}  \left(\frac{1}{\abs{F_N}}\sum_{g\in F_N} \norm{\pi(g)x}^2 \right) ,
\end{equation}
by taking the limit along $\mathcal{U}$ of the bounded sequence in \eqref{eq:SFCnorm} indexed by $N$. Using \eqref{eq:bc} we obtain
\begin{equation*}
m \norm{x} \le \abs{x} \le M\norm{x}
\end{equation*}
and thus $\abs{\cdot}$ is an equivalent norm. It is also a Hilbertian norm for the inner product
$$ \lim_{\mathcal{U}}  \left( \frac{1}{\abs{F_N}}\sum_{g\in F_N} \ps{\pi(g)x}{\pi(g)y} \right) .
$$
Let $s\in S$. We can write
\begin{equation}\label{eq:nisombis}
\abs{\pi(s)x}^2 -  \abs{x}^2 = \lim_{\mathcal{U}}  \left( \frac{1}{\abs{F_N}}\sum_{g\in F_N} \norm{\pi(gs)x}^2 -  \frac{1}{\abs{F_N}}\sum_{g\in F_N} \norm{\pi(g)x}^2 \right) .
\end{equation}
As above, we have
\begin{align*}
& \abs{\frac{1}{\abs{F_N}}\sum_{g\in F_N} \norm{\pi(gs)x}^2 -  \frac{1}{\abs{F_N}}\sum_{g\in F_N} \norm{\pi(g)x}^2} \\
&= \abs{\frac{1}{\abs{F_N}}\sum_{h\in F_Ns} \norm{\pi(h)x}^2 -  \frac{1}{\abs{F_N}}\sum_{g\in F_N} \norm{\pi(g)x}^2}\\
&= \abs{\frac{1}{\abs{F_N}}\sum_{h\in F_{N}s\setminus F_N} \norm{\pi(h)x}^2 -  \frac{1}{\abs{F_N}}\sum_{g\in F_{N}\setminus F_{N}s} \norm{\pi(g)x}^2}\\
&\le \frac{1}{\abs{F_N}} \sum_{g\in F_{N}s\triangle F_N} \norm{\pi(g)x}^2  .
\end{align*}
Using \eqref{eq:bd} we obtain that the limit in \eqref{eq:nisombis} is zero. Thus $\pi(s)$ is an isometry in the new norm for each $s\in S$. 
\end{proof}
As was the case for groups, the asymptotic condition \eqref{eq:bd} in Theorem~\ref{thm:SFC} can be removed under additional conditions on \F sets. We will see in the next section that 
in the case of the semigroup $S=\N$ and of the  \F sets $F_N = \{0, 1, \cdots, N-1\}$, 
the second half of the chain of inequalities \eqref{eq:bc} implies \eqref{eq:bd}. Although we have some partial results, presently we do not know a useful characterization of \F sets with this property for general semigroups. 
%We plan to return to this problem in the future.

\section{Operators similar to isometries}
\begin{proof}[Proof of Theorem \ref{thm:K}]
For $x,y\in \h$ the sequence of complex numbers
$$  \left( \frac{1}{N} \sum_{n=0}^{N-1}\ps{T^nx}{T^ny}\right)_{N\ge 1}
$$
is bounded. Indeed, condition \eqref{eq:1} and the Cauchy-Schwarz inequality imply
\begin{align*}
\abs{\frac{1}{N} \sum_{n=0}^{N-1}\ps{T^nx}{T^ny}} &\le \frac{1}{N} \sum_{n=0}^{N-1}\abs{\ps{T^nx}{T^ny}} \\
&\le \frac{1}{N} \sum_{n=0}^{N-1}\norm{T^nx}\norm{T^ny}\\
&\le \frac{1}{N} \left(\sum_{n=0}^{N-1}\norm{T^nx}^2\right)^{1/2}\, \left(\sum_{n=0}^{N-1}\norm{T^ny}^2\right)^{1/2}\\
& \le M\norm{x}\norm{y}.
\end{align*}
Let $\mathcal{U}$ be a non-principal ultrafilter on $\N$. Define an operator $F$ on $\h$ by
\begin{equation}\label{eq:ultraf}
\ps{Fx}{y} = \lim_{\mathcal{U}} \left( \frac{1}{N} \sum_{n=0}^{N-1}\ps{T^nx}{T^ny}\right).
\end{equation}
Then $F$ is a positive operator. Set $D=F^{1/2}$, the square root of $F$.

Let $h\in\h$. We have, using the Cauchy-Schwarz inequality and the second half of \eqref{eq:1},
\begin{align*}
\sum_{n=0}^{N-1}\norm{T^nh} &\le \sqrt{N}\left(\sum_{n=0}^{N-1}\norm{T^nh}^2\right)^{1/2}\\
 &\le N\sqrt{M}\norm{h} .
\end{align*}
Therefore
$$\frac{1}{N} \sum_{n=0}^{N-1}\norm{T^nh} \le \sqrt{M}\norm{h}$$
for every $h\in\h$, that is, $T$ is an \emph{absolutely Ces\`aro bounded} operator in the terminology of \cite{muller}. It follows from \cite[Theorem 2.4]{muller} that
\begin{equation}\label{zero}
\lim_{N\to\infty} \frac{\norm{T^N}}{\sqrt{N}} = 0.
\end{equation}
We have
\inserteq{\norm{Dx}^2 = \ps{Fx}{x} = \lim_{\mathcal{U}} \left( \frac{1}{N} \sum_{n=0}^{N-1} \norm{T^nx}^2\right)}{eq:newnorm}
and so
\begin{align*}
\norm{DTx}^2 - \norm{Dx}^2 & = \lim_{\mathcal{U}} \left( \frac{1}{N} \sum_{n=0}^{N-1} \norm{T^{n+1}x}^2 -  \frac{1}{N} \sum_{n=0}^{N-1} \norm{T^{n}x}^2\right)\\
& = \lim_{\mathcal{U}} \left( \frac{\norm{T^Nx}^2}{N} - \frac{\norm{x}^2}{N}\right) \\
& = \lim_{\mathcal{U}} \left( \frac{\norm{T^Nx}^2}{N}\right)\\
& = 0 \quad \textrm{ (using condition \eqref{zero})}.
\end{align*}
Hence $T$ is an isometry in the new norm \eqref{eq:newnorm}. As $\norm{Dx}^2 \le M\norm{x}^2$ it suffices to show that $D$ is invertible. Indeed, in this case \eqref{eq:newnorm} will be an equivalent, Hilbertian norm, and $T$ will be similar to an isometry.

We now undertake the proof that $D$ is invertible. Using the lower bound from \eqref{eq:1} we have
\inserteq{m\norm{Kx} \le \ps{Fx}{x} = \norm{Dx}^2.}{eq:FD}
Therefore
\inserteq{mK^*K \le D^2.}{eq:KD}
As the range $\R(K)$ of $K$ is closed, $0$ is an isolated point of $\sigma((K^*K)^{1/2})\cup \{0\}$. The inequality \eqref{eq:KD} implies that $0$ is an isolated point of $\sigma(D)\cup \{0\}$ and thus $\R(D)$ is closed. The relation \eqref{eq:FD} implies that the kernel $\n(D)$ of $D$ is included in the kernel $\n(K)$ of $K$. In particular, $\n(D)$ is finite dimensional.

Suppose that $\n(D) \neq \{0\}$ and let $x\in \n(D)\setminus\{0\}$. As $\norm{DTx}^2 = \norm{Dx}^2$ we obtain $T^nx \in \n(D)$ for every $n\ge 1$. Notice that $x \neq 0$ implies $T^nx \neq 0$ for every $n$. Indeed, $T$ is a one-to-one operator since $0$ cannot be an eigenvalue for $T$ because of the condition \eqref{eq:3}.

Now we prove by induction that the set $\{T^kx : 0 \le k\le n\}$ is linearly independent for every $n\ge 1$. Suppose that $\{ x, Tx \}$ is linearly dependent. Then there is $\lambda$ such that $Tx = \lambda x$ and, as $x\neq 0$, we have $\abs{\lambda} = 1$ by \eqref{eq:3}. Therefore, using the equality $T^jx = \lambda^j x$, we obtain
$$ 0 = \norm{DTx} = \lim_{\mathcal{U}} \left( \frac{1}{N} \sum_{n=0}^{N-1} \norm{T^{n}x}^2\right) = \norm{x}^2,$$
a contradiction. Thus $\{ x, Tx \}$ is linearly independent.

Suppose now that $\{T^kx : 0 \le k\le n\}$ is linearly independent for some $n\ge 1$ and that $\{T^kx : 0 \le k\le n+1\}$ is linearly dependent. Then
$$ \alpha_0 x + \alpha_1 Tx + \cdots + \alpha_n T^nx + T^{n+1} = 0$$
for some complex scalars $\alpha_0, \cdots , \alpha_n$ with at least one of them being non-zero. If $\alpha_0 = 0$, then, using the injectivity of $T$, we obtain a linear dependence relation between $x, Tx, \cdots , T^nx$ which contradicts the fact that $\{T^kx : 0 \le k\le n\}$ is linearly independent. Thus $\alpha_0 \neq 0$. Denote by $P$ the polynomial
$$P(z) = \alpha_0 + \alpha_1 z + \cdots + \alpha_n z^n + z^{n+1}$$
and let $\lambda$ be a root of $P$. Since $\alpha_0 \neq 0$ we have $\lambda \neq 0$. Consider now the vector
\begin{align*}
y &= T^n x + (\alpha_n + \lambda) T^{n-1}x + (\alpha_{n-1} + \alpha_n \lambda + \lambda^2)T^{n-2}x +\\
& \cdots +(\alpha_2 + \alpha_3 \lambda + \cdots + \alpha_n \lambda^{n-2} + \lambda^{n-1})Tx\\
& + (\alpha_1 + \alpha_2\lambda + \cdots + \alpha_n\lambda^{n-1} + \lambda^n)x .
\end{align*}
It follows from the definition of $y$ and from $P(\lambda) = 0$ that $Ty = y$. We have $y \neq 0$ since $\{T^kx : 0 \le k\le n\}$ is linearly independent. Hence $y$ is an eigenvector for the corresponding eigenvalue $\lambda$ and, using again condition \eqref{eq:3}, one must have $\abs{\lambda} = 1$. It follows from the definition of $y$ that $Dy = 0$. We obtain
$$0 = \norm{Dy} = \lim_{\mathcal{U}} \left( \frac{1}{N} \sum_{n=0}^{N-1} \norm{T^{n}y}^2\right) = \lim_{\mathcal{U}} \left( \frac{1}{N} \sum_{n=0}^{N-1} \abs{\lambda}^{n}\norm{y}^2\right) = \norm{y}^2,$$
a contradiction. Therefore $\{T^kx : 0 \le k\le n+1\}$ is linearly independent. By induction, $\{T^kx : 0 \le k\le n\}$ is linearly independent for every $n\ge 1$ and thus $\n(D)$ is a infinite dimensional subspace. This contradicts the hypothesis.

We finally obtain $\n(D) = \{0\}$. The self-adjoint operator $D$ is injective and has closed range, so it is invertible. Thus $DTD^{-1}$ is an isometry and the proof is complete.
\end{proof}
The following consequence was the initial motivation of our investigations.
\begin{corollary}\label{cor:ces}
Let $T\in\B$. Suppose that there are two positive constants $m$ and $M$ such that, for every $n\ge 1$ and every $x\in \h$,
\begin{equation}\label{eq:ces}
m^2\norm{x}^2 \le \frac{1}{N}\sum_{n=0}^{N-1}\norm{T^nx}^2 \le M^2\norm{x}^2 .
\end{equation}
Then $T$ is similar to an isometry.
\end{corollary}
\begin{proof}
Let $\lambda$ be an eigenvalue of $T$ verifying $Tx=\lambda x$ for a unit vector $x$.
Equation \eqref{eq:ces} implies
$$ m^2 \le \frac{1}{N}\sum_{n=0}^{N-1} \abs{\lambda}^{2n} \le M^2 .$$
Thus the sequence $\abs{\lambda}^{2n}$ does not go to zero or to infinity as $n$ tends to infinity. We obtain that any eigenvalue of $T$ is of modulus one. We apply Theorem \ref{thm:K} with $K=I$.
\end{proof}

\begin{remark} We can give a simple proof of Theorem \ref{thm:K}, not using limits along ultrafilters, for the particular case when $K=I$ (the identity operator) and $m=1$ in \eqref{eq:1} ($T$ is expansive).

\smallskip

Here is the proof. Let $T\in \B$ be an operator such that $T^*T\ge I$ and
$$\left\|\sum_{j=0}^n T^{*j}T^j\right\|=O(n) \quad \text{ as } n\to \infty .$$
Let us denote $A_n=\frac{1}{n+1} \sum_{j=0}^n T^{*j}T^j$ for $n\ge 1$. Since $T$ is expansive, we have $$I\le A_n\le T^{*n}T^n$$ and so
$$
A_{n+1}-A_n=\frac{1}{n+1}(A^{* (n+1)}A^{n+1}-A_{n+1})\ge 0.
$$
Hence the sequence $\{A_n\}_{n\ge 1}$ is increasing and, by hypothesis, bounded. Thus $\{A_n\}$ strongly converges to an operator $A$. As $A_n\ge I$ for $n\ge 1$, the limit operator $A$ satisfies $A\ge I$. In particular, $A$ is invertible. In addition, we obtain
\begin{eqnarray*}
T^*AT&=& s-\lim_{n\to \infty} T^*A_nT= s-\lim_{n\to \infty} \frac{1}{n+1} \sum_{j=1}^{n+1} A^{*j}A^j\\
&=& s-\lim_{n\to \infty} \left(\frac{n+2}{n+1}A_{n+1}-\frac{I}{n+1}\right)=A.
\end{eqnarray*}
Then the new norm $\abs{x}^2 = \norm{A^{1/2}x}^2 = \ps{Ax}{x}$ is an inner product norm which is equivalent to the original one since $A$ is invertible. With respect to this norm $T$ is an isometry since $T^*AT=A$.
\end{remark}

\bibliographystyle{amsalpha}

\end{document}